\documentclass{amsart}

\usepackage{helvet, color}

\usepackage{comment,amscd,amsmath,amsxtra,amsthm,amssymb,stmaryrd,xr,mathrsfs,mathtools,enumerate,url}
\usepackage[all,cmtip]{xy}
\usepackage{hyperref}
 \DeclareFontFamily{U}{wncy}{}
    \DeclareFontShape{U}{wncy}{m}{n}{<->wncyr10}{}
    \DeclareSymbolFont{mcy}{U}{wncy}{m}{n}
    \DeclareMathSymbol{\Sh}{\mathord}{mcy}{"58}

\newtheorem{theorem}{Theorem}[section]
\newtheorem{lemma}[theorem]{Lemma}

\newtheorem{proposition}[theorem]{Proposition}
\newtheorem{corollary}[theorem]{Corollary}
\newtheorem{definition}[theorem]{Definition}

\numberwithin{equation}{section}
\newtheorem{lthm}{Theorem}

\theoremstyle{remark}
\newtheorem{remark}[theorem]{Remark}

\newcommand{\fN}{\mathcal{N}}

\newcommand{\fR}{\mathfrak{R}}

\newcommand{\Tr}{\operatorname{Tr}}
\newcommand{\Gal}{\operatorname{Gal}}
\newcommand{\Fil}{\operatorname{Fil}}

\newcommand{\DD}{\mathbb{D}}

\newcommand{\fp}{\mathfrak{p}}

\newcommand{\fL}{\mathfrak{L}}

\newcommand{\QQ}{\mathbb{Q}}

\newcommand{\Qp}{\mathbb{Q}_p}
\newcommand{\Zp}{\mathbb{Z}_p}

\newcommand{\sH}{\mathscr{H}}

\newcommand{\Frac}{\mathrm{Frac}}
\newcommand{\ord}{\mathrm{ord}}
\newcommand{\GL}{\mathrm{GL}}
\newcommand{\cG}{\mathcal{G}}

\newcommand{\cris}{\mathrm{cris}}

\newcommand{\Z}{\mathbb{Z}}
\newcommand{\p}{\mathfrak{p}}
\newcommand{\Q}{\mathbb{Q}}
\newcommand{\F}{\mathbb{F}}
\newcommand{\C}{\mathbb{C}}

\newcommand{\vp}{\varphi}
\newcommand{\cL}{\mathcal{L}}

\newcommand{\cO}{\mathcal{O}}

\newcommand{\Iw}{\mathrm{Iw}}
\newcommand{\HIw}{H^1_{\mathrm{Iw}}}

\newcommand{\Dcris}{\mathbb{D}_{\rm cris}}
\newcommand{\Mlog}{M_{\log}}

\newcommand{\fF}{\mathfrak{F}}

\newcommand{\Hom}{\mathrm{Hom}}
\newcommand{\Sel}{\mathrm{Sel}}
\newcommand{\Char}{\mathrm{char}}

\newcommand{\pr}{\mathrm{pr}}

\renewcommand{\Col}{\mathrm{Col}}
\newcommand{\rank}{\mathrm{rank}}

\newcommand{\cX}{\mathcal{X}}

\newcommand{\ur}{\mathrm{ur}}

\newcommand{\TT}{\mathbf{T}}

\newcommand{\cyc}{\mathrm{cyc}}

\newcommand{\fP}{\mathfrak{P}}

\newcommand{\sL}{\mathscr{L}}
\newcommand{\sW}{\mathscr{W}}

\newcommand{\sF}{\mathscr{F}}

\newcommand{\KK}{k}
\newcommand{\ac}{\mathrm{ac}}
\newcommand{\uCol}{\underline{\Col}}
\newcommand{\cY}{\mathcal{Y}}

\begin{document}

\title{On the BDP Iwasawa main conjecture for modular forms}

\author[A.~Lei]{Antonio Lei}
\address[Lei]{Department of Mathematics and Statistics\\University of Ottawa\\
150 Louis-Pasteur Pvt\\
Ottawa, ON\\
Canada K1N 6N5}
\email{antonio.lei@uottawa.ca}

\author[L.~Zhao]{Luochen Zhao}
\address[Zhao]{Department of Mathematics\\ Johns Hopkins University\\ 404 Krieger Hall\\ 3400 N.~Charles Street\\ Baltimore, MD 21218\\ USA}
\email{lzhao39@jhu.edu}

\begin{abstract} Let $K$ be an imaginary quadratic field where $p$ splits, $p\geq5$ a prime number and  $f$ an eigen-newform of even weight and level $N>3$ that is coprime to $p$. Under the Heegner hypothesis, Kobayashi--Ota showed that one inclusion of the Iwasawa main conjecture of $f$ involving the Bertolini--Darmon--Prasanna $p$-adic $L$-function holds after tensoring by $\mathbb{Q}_p$. Under certain hypotheses, we improve upon Kobayahsi--Ota's result and show that the same inclusion holds integrally. Our result implies the vanishing of the Iwasawa $\mu$-invariants of several anticyclotomic Selmer groups. \end{abstract}


\maketitle

\section{Introduction}\label{section:intro}
%%%%%%%%%%%%%%%%%%%%%%%%%%%%%%%%%%%%%%%%%%%%%%%%
\subsection{Background and main results}
 Throughout this paper, $p \geq 5$ denotes a fixed prime number. We will consider a normalized eigen-newform of level $\Gamma_0(N)$ and weight $2r$, where $N>3$  and $r \geq 1$ are integers, and $p\nmid N\varphi(N)$. Here $\varphi$ is the Euler totient function. Let $K$ be an imaginary quadratic extension of $\Q$ with discriminant $d_K$ coprime to $Np$ in which every prime dividing $Np$ splits. We suppose that $p$ splits into $\fp$ and $\bar{\fp}$ in $K$ and assume that $K$ has unit group $\{ \pm1\}$ and that $d_k$ is odd or divisible by $8$. Furthermore, as in \cite[Theorem~1.5]{KobOta}, we assume that $p$ does not divide $6h_K$, where $h_K$ denotes the class number of $K$. Let $T_f$ be a Galois-stable lattice inside the Deligne representation attached to $f$. We assume that the image of $G_{\Q}\to \mathrm{Aut}_{\mathcal{O}}(T_f)\simeq \mathrm{GL}_2(\mathcal{O})$ contains the subgroup
\begin{align*}
    \{\sigma \in \mathrm{GL}_2(\mathbf{Z}_p) \mid \det(\sigma) \in (\mathbf{Z}_p^\times)^{2r-1}\}.
\end{align*}

Fix an embedding $K \hookrightarrow \C$ as well as an embedding $\iota_p:\overline{\Q} \hookrightarrow \C_p$
 such that $\p$ lands inside the maximal ideal of $\cO_{\C_p}$, the ring of integers of $\C_p$.
Let $\fF=\Q_p(\{a_n(f):n\ge1\},\alpha,\beta)$, where $\alpha$ and $\beta$ are the roots of $X^2-a_p(f)X+p^{2r-1}$ and let $\cO$ denote its valuation ring with a fixed uniformizer $\varpi$. 
 Denote by $K_\infty$ the anticyclotomic $\Z_p$-extension of $K$ with Galois group $\Gamma$. Let $\Lambda$ denote the Iwasawa algebra $\cO[[\Gamma]]$. Throughout we will fix a topological generator $\gamma\in \Gamma$, and identify $\Lambda$ with the power series ring $\cO[[X]]$.

Let $\mathfrak{L}(f/K)\in \Lambda^{\mathrm{ur}}$ denote the $p$-adic $L$-function attached to $f$ over $K$ defined by Brako\v{c}evi\'{c} and Bertolini--Darmon--Prasanna in \cite{miljan,BDP}, where $\Lambda^{\mathrm{ur}}$ is the Iwasawa algebra $\Lambda\otimes_{\cO}\hat\cO^\ur$. Here $\hat\cO^\ur$ denotes the $p$-adic completion of the ring of integers of the maximal unramified extension of $\fF$.
The square of $\fL(f/K)$  interpolates the $L$-values of $f$ base changed to $K$ twisted by certain anticyclotomic characters and is of significant arithmetic importance.

The Iwasawa main conjecture predicts that  $\fL(f/K)^2$ generates the characteristic ideal of the Pontryagin dual of certain Selmer group attached to $f$ over $K_\infty$, which we denote by $\cX_{(\emptyset,0)}(f)$. Kobayashi--Ota proved one inclusion of this conjecture after tensoring by $\Qp$. More precisely, they proved the following theorem:

\begin{theorem}[{\cite[Theorem~1.5]{KobOta}}]\label{thm:KO}
Under the notation and hypotheses above, the $\Lambda$-module $\mathcal{X}_{(\emptyset,0)}(f) $ is torsion. Furthermore, 
\[
\fL(f/K)^2\in \Char_\Lambda\left(\mathcal{X}_{(\emptyset,0)}(f)\right)\otimes_\Lambda \Lambda^\ur\otimes_{\Zp}\Qp.
\]
\end{theorem}

When $f$ is ordinary at $p$ and $r=1$, Howard has proved one inclusion of  the Heegner point main conjecture (see \cite[Theorem~B]{howard} and \cite[Theorem~B]{howard2}). As explained in \cite[proof of Theorem~3.4]{Castella} (see in particular equation (A.9) on P.178 in \textit{op. cit.}), this implies the inclusion in Theorem~\ref{thm:KO}, without tensoring by $\Qp$. See also \cite[Theorem~5.2]{BCK} where this is discussed further. When $f$ is non-ordinary at $p$, this question is studied in the preprint \cite{CCSS} (see Theorem 5.8 \textit{loc.~cit.~}in particular) using a different approach under different hypotheses. For more general $r$, one may deduce the same inclusion by the same argument upon replacing Howard's result by \cite[Theorem~1.1]{LongoVigni}. The goal of this paper is to prove the following integral version of Theorem~\ref{thm:KO}.
\begin{lthm}\label{thmA}
Let $K\ne \Q(\sqrt{-1}),\Q(\sqrt{-3})$ be an imaginary quadratic field with discriminant $d_K$ odd or divisible by $8$. Let $N>3$ be an integer such that each prime dividing $N$ splits in $K$. Let $f$ be a weight 2 normalized eigen-newform of level $\Gamma_0(N)$, and $T_f$ be a Galois-stable lattice inside the attached Deligne representation of $f$. Let $p$ be a prime such that
\begin{enumerate}
    \item $p$ splits in $K$;
    \item $p\nmid 6h_K N\varphi(N)$;
    \item $f$ is non-ordinary at $p$ and $\ord_p(a_p(f))\ge\frac{1}{p+1}$, where $\ord_p$ is the $p$-adic valuation on $\cO$ normalized by $\ord_p(p)=1$;
    \item The image of $G_{\Q}\to \mathrm{Aut}_{\cO}(T_f)\simeq \GL_2(\cO)$ contains $\GL_2(\Z_p)$.
\end{enumerate}
Then the inclusion
\[
\fL(f/K)^2\in \Char_\Lambda\left(\mathcal{X}_{(\emptyset,0)}(f)\right)\otimes_\Lambda \Lambda^\ur
\]
holds.
\end{lthm}

We recall that \cite[Theorem~B]{Hsieh-Doc} says that the $\mu$-invariant of $\fL(f/K)$ is zero (meaning that it does not belong to $\varpi\Lambda^\ur$). 
Therefore, whenever Theorem~\ref{thmA} holds, $\mathcal{X}_{(\emptyset,0)}(f)$ is a finitely generated $\cO$-module and thus  $\mu(\fL(f/K)) = \mu(\cX_{(\emptyset,0)}(f))=0$ (interested readers might wish to compare this with \cite[Theorems~1.1 and 1.2]{pollackweston11}, where similar results in the definite setting have been obtained).
Furthermore, Theorem~\ref{thmA} implies the validity of the hypothesis \textbf{(H-$\subseteq$)} in \cite[\S4]{HL}. In particular, on combining Theorem~5.1 of \textit{op. cit.} with Theorem~\ref{thmA}, we obtain new instances where the $\mu$-invariants of several anticyclotomic Selmer groups, including signed Selmer groups attached to a $p$-non-ordinary modular form, vanish. If $f$ corresponds to a $p$-supersingular elliptic curve $E$ defined over $\QQ$ (in which case $a_p(E)$ is always zero since we have assumed that $p\ge5$), this recovers the last assertion of \cite[Theorem~B]{matar2021} on the $\mu$-invariants of  the anticyclotomic plus and minus Selmer groups attached to $E$.

While Theorem~\ref{thmA} assumes that $r=1$, this condition can in fact be relaxed via a straightforward application of \cite[Theorem~4.3]{HL-MRL}\footnote{Even though the hypothesis that the modular forms are $p$-ordinary of square-free level is in force in \cite{HL-MRL}, the proof of \textit{loc. cit.} is in fact purely algebraic and can be carried out in verbatim to our setting. Note in particular that for the non-ordinary case,  the local conditions of the Selmer group $\mathcal{X}_{(\emptyset,0)}(f)$  are independent of the reduction type. See also \cite[Theorem~A]{LMX}.}, which tells us that given two modular forms $f$ and $g$ with isomorphic residual representations modulo $\varpi$, the $\mu$-invariant of $\cX_{(\emptyset,0)}(f)$ vanishes if and only if that of $\cX_{(\emptyset,0)}(g)$ does. Therefore, it allows us to deduce the following:
\begin{lthm}\label{thmB}
    Let $f$ be a modular form satisfying the hypotheses in Theorem~\ref{thmA}. Let $g$ be a modular form satisfying the hypotheses of Theorem~\ref{thm:KO} (in particular $g$ satisfies the Heegner hypothesis with respect to $K$) and that the residual representation of $g$ modulo $\varpi$ is isomorphic to that of $f$. Then
\[
\fL(g/K)^2\in \Char_\Lambda\left(\mathcal{X}_{(\emptyset,0)}(g)\right)\otimes_\Lambda \Lambda^\ur.
\]
\end{lthm}

\subsection{Strategy of proof}

In \cite{howard}, Howard built on the machinery of Kolyvagin systems developed by Mazur--Rubin \cite{MazurRubin} to prove one inclusion of the Heegner point main conjecture for $p$-ordinary elliptic curves. In order to compare the $\mu$-invariants of the two $\Lambda$-ideals  appearing in the conjecture, Howard considered the localization  of these  ideals at height-one primes of $\Lambda$ of the form $(X^m+\varpi)$  as $m$ varies.
Our proof of Theorem~\ref{thmA} (which is carried out in \S\ref{sec:proof}) follows a similar strategy, building on several refinements of the proof carried out in \cite{KobOta} and utilizing the height-one primes of the form $(\Theta_m)$, where $\Theta_m$ is the minimal polynomial of $\zeta_{p^m}-1$ and $\zeta_{p^m}$ is a primitive $p^m$-th root of unity. The reason why we consider this family instead of the one considered by Howard is that we need to analyze the integrality of specialization of the Perrin-Riou map at these height-one primes. Our choice of height-one primes is well-adapted for these calculations  since we are able to employ previous works on the Perrin-Riou map (e.g. \cite{leiloefflerzerbes15,LL-IJNT}). 
We give a brief summary of the main steps carried out in our proof of Theorem~\ref{thmA}.
\begin{itemize}    
\item We improve the bounds appearing in the control theorem \cite[Proposition 5.7]{KobOta}, so as to make Howard's argument work for the family of prime ideals $\{(\Theta_m)\}_{m\ge 1}$ in $\Lambda$. This is carried out in \S\ref{sec:control}.
    
    \item In \S\ref{sec:bp}, we analyze carefully the $p$-adic valuation of the constant $b_{\fP}$ specified by Proposition 4.15 \textit{ibid.} (see also Proposition~\ref{prop:KO-length}, where we review this result). As the reader will notice, such analysis only becomes necessary when one studies the $\mu$-invariant of the corresponding torsion $\Lambda$-modules. In fact, the family of prime ideals we employ are, in certain sense, orthogonal to the family used by Kobayashi--Ota to prove Theorem~\ref{thm:KO}. For our purposes, we prove a novel result (Proposition~\ref{prop:b_P}) to describe the behaviour of the constants $b_{\fP}$.

    \item We analyze in Appendix~\ref{appendix} integrality properties of the Perrin-Riou map specialized at our chosen primes $(\Theta_m)$ using the theory of integral Coleman maps. This allows us to give a precise estimate of the aforementioned constants $b_{\fP}$ and forms one of the key ingredients in the proof of Theorem~\ref{thmA}. This is also where the hypothesis that $r=1$ is crucially used. When $r=1$, we are able to describe the Perrin-Riou map using an explicit matrix. It may be possible to carry out our calculations for more general $r$ by seeking an appropriate generalization of  our description of the images of the Coleman maps given in Corollary~\ref{cor:image-basis-col}.
    \end{itemize}

\subsection*{Acknowledgement}We thank Ashay Burungale, Kazim Buyukboduk, Jeffrey Hatley, Chan-Ho Kim,  Katharina M\"uller  and Jiacheng Xia for interesting discussions during the preparation of the article. We are  grateful to Ming-Lun Hsieh for answering our questions regarding generalized Heegner cycles. We are also indebted to the anonymous referee for their valuable comments and suggestions. The first named author's research is supported by the NSERC Discovery Grants Program RGPIN-2020-04259 and RGPAS-2020-00096.

\section{Galois representations and Selmer groups}\label{S:Gal-Sel}
The purpose of this section is to review the notation that will be utilized in the remainder of the article. The notation and hypotheses  in the introduction will remain in force throughout. Furthermore, we assume that the hypotheses of Theorem~\ref{thmA} hold. Note in particular that $f$ is assumed to be of weight $2$.

Let $W_f$ be Deligne's two-dimensional $\fF$-linear $G_\QQ$-representation attached to $f$ constructed in \cite{deligne69}.  If we normalize such that the cyclotomic character has Hodge-Tate weight $+1$, then $W_f$ has Hodge-Tate weights $\{0,1\}$. Let $R_f$ be the Galois-stable $\mathcal{O}$-lattice in $W_f$ as in \cite[§3]{nekovar1992} (where it is denoted by $A$).
Write $T_f= \Hom(R_f,\cO)$, $V_f=T_f\otimes_{\Zp}\Qp$ and $A_f=V_f/T_f$. Note that our $A_f$ corresponds to the representation $W$ in \cite[p566]{KobOta}.
  For notational simplicity we often drop the subscript $f$ from $T_f$, $V_f$ and $A_f$.

We also recall some notation from \cite[\S4.4]{KobOta}. In what follows $\mathfrak{P}$ shall always denote a height-one prime of $\Lambda$ not equal to $\varpi\Lambda$. Denote by $S_{\mathfrak P}$ the integral closure of $\Lambda/\mathfrak{P}$ in the fractional field; thus $S_{\mathfrak P}$ is a discrete valuation ring that is finite over $\Zp$. In what follows, for each $\mathfrak P$ as above, we fix an $\mathcal{O}$-linear embedding $\iota_{\mathfrak P}: \mathrm{Frac}(S_{\mathfrak P}) \to \overline{\Qp}$, and denote by $\chi_{\mathfrak P}$ the character $\Gamma \to S_{\mathfrak P}^\times \xrightarrow{\iota_{\mathfrak P}} \overline{\Qp}^\times$, as well as its composition with the projection $G_K\twoheadrightarrow \Gamma$ by an abuse of notation. The function on $\Lambda$ induced by $\chi_\fP$ will also be denoted by the same symbol. The ring $S_{\mathfrak P}$ will also be regarded as a $G_K$-module, where the action is given by the natural map $G_K\to \Lambda^\times \to S_{\mathfrak P}^\times$. For $M\in \{T,V,V/T\}$, write $M_{\mathfrak{P}} = M\otimes_{\mathcal O} S_{\mathfrak P}$, with a diagonal action by $G_K$.

We write $\Sel_{(\emptyset,0)}(K_\infty,A_f) $ for the Selmer group defined as in \cite[Definition~3.5]{HL}. Its Pontryagin dual is denoted by $\mathcal{X}_{(\emptyset,0)}(f) $. We recall from \cite[Proposition~2.9]{KobOta} that while the definition of $\mathcal{X}_{(\emptyset,0)}(f)$ depends on the choice of the lattice $R_f$, its $\Lambda$-characteristic ideal is well-defined.

	%First, we turn to Selmer groups, whose detailed definitions may be found in \cite[§2.3]{KobOta} and \cite[§3]{HL}. 
	Put $\mathbf{T} = T\otimes_{\mathcal O} \Lambda$, $\mathbf{A} = \Hom_{\rm cont}(\mathbf{T},\mu_{p^\infty})$ and $K_n\subset K_\infty$ the subextension with $\Gal(K_n/K)\simeq \Z/p^n$. Then by Shapiro's lemma, we have $H^1(K,\mathbf{A})\simeq H^1(K_\infty,A)=\varinjlim_n H^1(K_n,A)$, so we can define a Selmer structure $\mathscr{F}_\Lambda$ on $H^1(K,\mathbf{A})$ such that $H^1_{\mathscr{F}_\Lambda}(K,\mathbf{A}) \simeq \Sel_{(\emptyset,0)}(K_{\infty},A)$. We have thus 
	\begin{align*}
		\mathcal{X}_{(\emptyset,0)}(f) = \Sel_{(\emptyset, 0)}(K_{\infty},A)^\vee \simeq H^1_{\mathscr{F}_{\Lambda}}(K,\mathbf{A})^\vee.
	\end{align*}
    More generally, if $M\in \{T,V,V/T\}$, we can define Selmer structures on $M_\fP$ similarly; we refer the reader to \cite[\S 2.3]{KobOta} for details. 
 
	We mention two additional isomorphisms that will be used in \S\ref{sec:proof}. Let $\tau\in G_K$ denote the complex conjugation. It induces an involution $\iota: \Lambda\to \Lambda$ sending $g\in \Gamma$ to $g^{-1}$.
	
	\begin{itemize}
		\item[(a)] We have
		\begin{align*}
			H^1_{(\emptyset, 0)}(K,A_{\mathfrak P}) \simeq H^1_{(0,\emptyset)}(K,A_{\mathfrak{P}^\iota}),
		\end{align*}
		due to the identification $H^1(K_{\mathfrak p},A_{\mathfrak P})\simeq H^1(K_{\bar{\mathfrak p}},A_{\mathfrak{P}^\iota})$ by $\tau$-twisting.
		
		\item[(b)] We have
		\begin{align*}
			(H^1_{\mathscr{F}_\Lambda}(K,\mathbf{A})[\mathfrak P])^\vee \simeq \mathcal{X}_{(\emptyset,0)}(f)/\mathfrak{P}^\iota,
		\end{align*}
		which follows from the definition of the contragredient action: If $M$ is a $\Lambda$-module and $M^\vee = \Hom(M,\Q_p/\Z_p)$ its Pontryagin dual, then the contragredient action of $\Lambda$ on $M^\vee$ is given by
        \begin{align*}
            g\cdot x = x\circ\iota(g)
        \end{align*}
        for all $g\in \Lambda$ and $x\in M^\vee$.
	\end{itemize}
	
	\vspace{3mm}

\begin{lemma}\label{lem:length}
		Let $M$ be a finitely generated $S_{\mathfrak P}$-module. Then
		\begin{align*}
			\mathrm{length}_{\mathcal O}(M) = f_{\mathfrak P}\cdot \mathrm{length}_{S_{\mathfrak P}}(M),
		\end{align*}
		where $f_{\mathfrak P}$ is the degree of the residual field extension of $S_{\mathfrak P}/\mathcal O$.
	\end{lemma}
	
	\begin{proof}
		Take a composition series $M=M_0\supsetneq M_1\cdots \supsetneq M_k=0$ of $S_{\mathfrak P}$-modules. Then $M_i/M_{i+1}\simeq S_{\mathfrak P}/\mathfrak n$ for all $0\le i< k$, where $\mathfrak n$ is the maximal ideal $S_{\mathfrak P}$. It follows that for all $0\le i<k$, we have a sequence of $\mathcal{O}$-submodules $M_i = M_{i,0}\supsetneq M_{i,1}\cdots\supsetneq M_{i,f_{\mathfrak P}} = M_{i+1}$ with $M_{i,j}/M_{i,j+1}\simeq \mathcal{O}/\varpi$ for all $0\le j<f_{\mathfrak P}$. Thus the sequence $\{M_{i,j}\}_{i,j}$ form a composition series of $M$ as an $\mathcal O$-module, whereby the lemma follows.
	\end{proof}

  From now on, we fix a generator $\sL$ of the ideal $\mathfrak{L}(f/K)\Lambda^\ur\cap\Lambda$.	We have:
	\begin{proposition}\label{prop:KO-length}	Suppose that $\chi_{\mathfrak P}(\mathscr{L})\ne 0$. Then the $S_{\mathfrak P}$-module $$\cX_{(0,\emptyset)}(K,A_{\mathfrak P})=H^1_{(0,\emptyset)}(K,A_{\mathfrak P})^\vee$$ is torsion and thus finite. Furthermore, there exists an explicit constant $b_{\mathfrak P}\in \Frac(S_\fP)^\times$ depending on $\iota_{\mathfrak P}$, such that $b_{\mathfrak P}\chi_{\mathfrak P}(\mathscr{L})\in S_{\mathfrak P}$ and
		\begin{align}
			\label{length-inequality}
			2 \cdot {\rm length}_{\mathcal O}\left(\frac{S_{\mathfrak{P}}}{ b_{\mathfrak P}\chi_{\mathfrak P}(\mathscr{L})}\right) \ge {\rm length }_{\mathcal O}(\cX_{(0,\emptyset)}(K,A_{\mathfrak P})).
		\end{align}
	\end{proposition}
	
	\begin{proof}
		This follows from combining  \cite[Lemma~5.3, Proposition 4.15, and Proposition 5.4]{KobOta} with Lemma~\ref{lem:length} above. Note that the statement in Proposition 4.15 \textit{loc.~cit~}should really be $b_{\fP}\in \Frac(S_\fP)^\times$ instead of $\fF^\times$ (see also the remark following \eqref{definition:b_P}).
	\end{proof}

\section{A control theorem}\label{sec:control}

In this section, we will mainly be interested in refining the control theorem \cite[Proposition~5.7]{KobOta} in the case where the height-one prime $\mathfrak{P}$ ranges in a special family $\{\mathfrak{P}_m\}_{m\ge 1}$ defined below. In particular, we prove a control theorem relating $\cX_{(0,\emptyset)}(K,A_{\mathfrak{P}^\iota})\simeq \cX_{(\emptyset,0)}(K,A_{\mathfrak{P}})$ to $\mathcal{X}_{(\emptyset,0)}(f)/\fP^\iota$  for such $\fP$. 

\vspace{3mm}

We start by defining $\fP_m$: For $m\ge 1$ let $\zeta_{p^m}$ be a fixed primitive $p^m$-th root of unity with $\zeta_{p^{m+1}}^p=\zeta_{p^m}$, and form the irreducible polynomial
\begin{align*}
	\Theta_m(X) = \prod_{\sigma\in \Gal(\fF(\mu_{p^m})/\fF)} (X+1-\zeta_{p^m}^\sigma) \in \fF[X].
\end{align*}
Then $\fP_m = (\Theta_m)\subset \Lambda$ defines a height-one prime ideal, and we choose the attached character $\chi_{\fP_m}$ to be such that $\chi(\gamma) = 1+\varpi_{\fP_m} := \zeta_{p^m}$. Note that $\Theta_m$ is also a distinguished polynomial. Next, we prove two preliminary lemmas.

\begin{lemma}
	\label{lem:rank-equal}
	Let $\tilde{\gamma}\in G_{K_{\mathfrak{p}}}$ be a lift of $\gamma\in \Gamma$. There exists $m_0\in \mathbf{Z}_{>0}$ such that for $m\ge m_0$ and $\mathfrak{P}=\mathfrak{P}_m$, $\mathrm{rank}_{\mathcal{O}}(T_{\mathfrak P}) = \mathrm{rank}_{\mathcal{O}}((\tilde{\gamma}-1)T_{\mathfrak P})$.
\end{lemma}

\begin{proof}
	In what follows, for a finite extension $S/\Z_p$, a given $S[[\tilde{\gamma}]]$-module $M$ and $a\in S$, we put 
	$$M^{\tilde{\gamma}=a}=\{x\in M: \tilde{\gamma}\cdot x=ax\}.$$ 
	Consider the short exact sequence
	\begin{align*}
		0\to T_{\mathfrak P}^{\tilde{\gamma} = 1} \to T_{\mathfrak P} \to (\tilde{\gamma}-1)T_{\mathfrak P} \to 0
	\end{align*}
	with the quotient map given by $(\tilde{\gamma}-1)$-multiplication. If we can show $T_{\mathfrak P}^{\tilde{\gamma} = 1} = 0$, then we would obtain
	\begin{align*}
		T_{\mathfrak P} \hookrightarrow (\tilde{\gamma}-1)T_{\mathfrak P} \subseteq T_{\mathfrak P},
	\end{align*}
	which would imply that $\mathrm{rank}_{\mathcal{O}}(T_{\mathfrak P}) = \mathrm{rank}_{\mathcal{O}}((\tilde{\gamma}-1)T_{\mathfrak P})$ as desired. 
	
	To prove the vanishing of $T_{\mathfrak P}^{\tilde{\gamma} = 1}$, let $R$ be a finite flat module over $\mathcal{O}$ containing all roots of $\Theta_m(X)$ with the trivial Galois action. Then
	\begin{align*}
		V_{\mathfrak P}^{\tilde{\gamma}=1}\otimes_{\mathcal{O}} R = (V\otimes_{\mathcal{O}} S_{\mathfrak P} \otimes_{\mathcal{O}} R)^{\tilde{\gamma} =1} \simeq \bigoplus_{\pi: \Theta_m(\pi) = 0} \left(V \otimes_{\mathcal{O}} R[X]/(X-\pi)\right)^{\tilde{\gamma} = 1}.
	\end{align*}
	For any given $\pi$ in the summand above, there is an isomorphism of vector spaces
	\begin{align*}
		\left(V \otimes_{\mathcal{O}} R[X]/(X-\pi)\right)^{\tilde{\gamma} = 1} \simeq V_R^{\tilde{\gamma} = (1+\pi)^{-1}}:= \{v\in V\otimes_{\mathcal{O}} R: \tilde{\gamma}\cdot v = (1+\pi)^{-1}v\}.
	\end{align*}
	Since $V_R^{\tilde{\gamma} = (1+\pi)^{-1}} \ne 0$ if and only if $(1+\pi)^{-1}$ is an eigenvalue of $\tilde{\gamma}$ acting on $V$, we see that for $m$ large enough,  $V_R^{\tilde{\gamma} = (1+\pi)^{-1}}$ vanishes for all $\pi$. Thus $V_{\mathfrak P}^{\tilde{\gamma} =1} = 0$, and the same is true for $T_{\mathfrak P}^{\tilde{\gamma} = 1}$.
\end{proof}

\begin{lemma}
	\label{lem:bound-torsion}
	Suppose $\mathfrak{P} = \mathfrak{P}_m = (\Theta_m(X))$ for some $m\ge 1$. The cardinality $|H^1(K_{\mathfrak p},T_{\mathfrak{P}^\iota})_{\rm tors}|$ is bounded by a constant independent of $m$.
\end{lemma}
\begin{proof}
	First note that $S_{\mathfrak{P}} = \Lambda/\mathfrak{P}$. The long exact sequence of $G_{K_\fp}$-cohomology arising from the  the tautological  short exact sequence $0\to T\to V\to A\to 0$ gives the isomorphism $H^1(K_{\mathfrak p},T_{\mathfrak{P}^\iota})_{\rm tors} \simeq A_{\mathfrak{P}^\iota}^{G_{K_{\mathfrak p}}}$. Let $\tilde{\gamma}\in G_{K_{\mathfrak p}}$ be a lift of $\gamma$ and $G_{\tilde{\gamma}}\subseteq G_{K_{\mathfrak p}}$ be the topological closure of the subgroup generated by $\tilde{\gamma}$. It suffices to show that $\left|A_{\mathfrak{P}^\iota}^{G_{\tilde{\gamma}}}\right|$ is bounded independently of $m$. By duality, we may identify $T_{\mathfrak P} = \mathrm{Hom}(A_{\mathfrak{P}^\iota},\mu_{p^\infty})$, whereby
	\begin{align*}
		(T_{\mathfrak P})_{G_{\tilde{\gamma}}}=T_{\mathfrak P}/\{\tilde{\gamma}^i t-t:i\in \mathbf{Z}_p,t\in T_{\mathfrak P}\}= T_{\mathfrak P}/(\tilde{\gamma}-1)T_{\mathfrak P} = \mathrm{Hom}\left(A_{\mathfrak{P}^\iota}^{G_{\tilde{\gamma}}},\mu_{p^\infty}\right).
	\end{align*}
	Thus,  it suffices to show that $|T_{\mathfrak P}/(\tilde{\gamma}-1)T_{\mathfrak{P}}|$ is bounded independently of $m$.
	
	Let $k=k(m)=\#\Gal(\fF(\mu_{p^m})/\fF)$. Consider the decomposition 
	$$S_{\mathfrak P} = \oplus_{0\le i<k} \mathcal{O} \cdot X^i.$$
	The action of $\tilde\gamma-1$ on the basis $1,X,\ldots X^{k-1}$ is given by: $(\tilde{\gamma}-1)X^j = X^{j+1}$ for $0\le j<k-1$ and $(\tilde{\gamma}-1)X^{k-1} = X^k = (-1)^{k+1}\mathrm{Nm}_{\fF(\mu_{p^m})/\fF}(\varpi_{\fP_m})-t_1 X - \cdots - t_{k-1}X^{k-1}$ for some $t_1,\cdots,t_{k-1}\in \fF$. Thus, the matrix representing the action of $\tilde{\gamma}-1$ with respect to this basis is given by
	\begin{align*}
		A = \left[
		\begin{matrix}
			0 & 0 & \cdots & 0 & \pm\mathrm{Nm}_{\fF(\mu_{p^m})/\fF}(\varpi_{\fP_m})\\
			1 & 0 & \cdots & 0 & -t_1\\
			0 & 1 & \cdots & 0 & -t_2\\
			\vdots & \vdots & \ddots & \vdots & \vdots\\
			0 & 0 & \cdots & 1 & -t_{k-1}
		\end{matrix}
		\right].
	\end{align*}
	Therefore,
	\begin{align*}
		\textstyle
		\det_{T\otimes_{\mathcal{O}} S_{\mathfrak P}}(\tilde{\gamma}-1) = \det_T(\tilde{\gamma}-1)\det(A) = \pm\mathrm{Nm}_{\fF(\mu_{p^m})/\fF}(\varpi_{\fP_m}) \det_T(\tilde{\gamma}-1).
	\end{align*}
	Let $m_0\ge 1$ be such that $\Q(\mu_{p^{m_0}})\supseteq \fF\cap \Q(\mu_{p^\infty})$. Then for $m\ge m_0$, we have
	\begin{align*}
		\mathrm{Nm}_{\fF(\mu_{p^m})/\fF}(\varpi_{\fP_m}) &= \mathrm{Nm}_{\fF(\mu_{p^{m_0}})/\fF}\circ \mathrm{Nm}_{\Q(\mu_{p^m})/\Q(\mu_{p^{m_0}})}(\zeta_{p^m}-1)\\
		&=\mathrm{Nm}_{\fF(\mu_{p^{m_0}})/\fF}(\zeta_{p^{m_0}}-1)\ne 0.
	\end{align*}
	It follows from Lemma \ref{lem:rank-equal} that $(\tilde{\gamma}-1)T_{\mathfrak P}$ is of the same rank as $T_{\mathfrak P}$ when $m$ is large enough. Thus, $\det_{T_{\mathfrak P}}(\tilde{\gamma}-1)\ne 0$ and $\det_{T}(\tilde{\gamma}-1)\ne 0$. Therefore, by the theory of elementary divisors,
	\begin{align*}
		\textstyle
		|T_{\mathfrak P}/(\tilde{\gamma}-1)T_{\mathfrak{P}}| = [\mathcal{O}:\det_{T\otimes_{\mathcal{O}} S_{\mathfrak P}}(\tilde{\gamma}-1)\mathcal{O}] = [\mathcal{O}:\mathrm{Nm}_{\fF(\mu_{p^{m_0}})/\fF}(\varpi_{\fP_{m_0}}) \det_T(\tilde{\gamma}-1)\mathcal{O}],
	\end{align*}
	which is clearly bounded independently of $m$, as desired.
\end{proof}

\begin{remark}
    Under the assumption that $f$ is non-ordinary at $p$, which is one of the running assumptions in Theorem \ref{thmA}, we shall see in the proof of Lemma \ref{lem:projection} below that $H^1(K_\fp,T_{\fP^\iota})_{\rm tors}=0$.
\end{remark}
In order to state our control theorem, we consider the natural map
\begin{align*}
	s_{\mathfrak P}: H^1_{(\emptyset,0)}(K,A_{\mathfrak P}) \to H^1_{\mathscr{F}_\Lambda}(K,\mathbf{A}[\mathfrak P]) \simeq H^1_{\mathscr{F}_\Lambda}(K,\mathbf{A})[\mathfrak P]
\end{align*}
(for the last isomorphism, see \cite[Lemma 3.5.3]{MazurRubin}).
\begin{proposition}
	\label{prop.control-theorem}
	There exists a finite set $\Sigma$ of height-one primes of $\Lambda$, containing $\varpi\Lambda$, such that for $\mathfrak{P}\notin \Sigma$, the map $s_{\mathfrak P}$ is a quasi-isomorphism. Moreover, for $\mathfrak{P}$ ranging in $\{\mathfrak{P}_m = (\Theta_m)\}_{m\ge 1}$, the $\mathcal O$-lengths of the kernel and the cokernel of $s_{\mathfrak P}$ are bounded by a constant $C$ depending only on $T_f$.
\end{proposition}

\begin{proof}
	Since $S_{\mathfrak{P}} = \Lambda/\mathfrak{P}$ for $\mathfrak{P}=\mathfrak{P}_m$, if we allow the constant $C$ to be  dependent on $\rank_{\mathcal{O}}(S_{\mathfrak{P}})$, then the affirmation of the proposition is given by \cite[Proposition 5.7]{KobOta}. As one can check, the possible dependence on $\mathrm{rank}_{\mathcal O}(S_{\mathfrak P})$ \textit{loc.~cit.~}comes from the bound on the length of $H^1(K_{\mathfrak p},T_{\mathfrak{P}^\iota})_{\rm tors}$ in Lemma 5.6 \textit{ibid}. If we replace this bound by the one given in Lemma \ref{lem:bound-torsion}, our result follows.
\end{proof}

\section{Studying the constant $b_\fP$}\label{sec:bp}
In this section,  we study the constant $b_\fP$ appearing in Proposition~\ref{prop:KO-length} for the ideals $\fP_m=(\Theta_m(X))$ as $m$ varies. The constant $b_\fP$ is defined in \cite[Proposition~4.15]{KobOta} in terms of three constants $q_\fP$, $\beta_\fP$ and $b_\fP'$, which we review below.

We begin with the constant $q_\fP$. It is defined right after Remark 4.4 of \textit{op. cit.} and is related to the constant $B(\chi)$  studied in Lemma 3.2 of \textit{op. cit.}

\begin{definition}
%Let $\rho$ be the anticyclotomic character $\psi_\fp\psi_{\overline\fp}^{-1}:\Gamma\rightarrow 1+p\Zp$ defined as in \cite[\S4.3]{KobOta}.
For  a character $\chi$ of $\Gamma$, we define the following constants.
\item[i)] We define $B(\chi)$ to be 
\begin{align*}
    \inf_{n} \ord_p\left(\frac{(\chi(\gamma)-1)^n}{n}\right).
\end{align*}
(See Lemma~3.2 of \textit{op. cit.} for the general definition of $B(\chi)$ for a given integer $h\ge1$. At the beginning of \S4.3 of \textit{op. cit.}, this integer is taken to be $2r-1$, which is $1$ in our case. This is why the character $\rho^i$ does not make an appearance in our definition.)
\item[ii)] Let $\alpha$ be a root of $X^2-a_p X+p$. We define
\[
q_\chi=p^{C}\cdot \alpha^{\max\{0,1-B(\chi)\}},
\]
where  $C\ge0$ is a constant that eliminates the denominators of the generalized Heegner classes with respect to the lattice $T_f$ (see \cite[\S9]{Kob}).
\end{definition}

\begin{lemma}\label{lem:q}
     Let $\chi=\chi_{\fP_m}$ for some $m\ge 1$. Then $\ord_p(q_\chi)=\ord_p(\alpha)\left(m -\frac{1}{p-1}\right)$.
\end{lemma}
\begin{proof}
   As explained in \cite[§4.2]{CastellaHsieh}, the Heegner classes already live in the cohomology groups with coefficients in our choice of the lattice $T$ (see the displayed equation (4.2) in \textit{op.~cit.}). In particular, the constant $C$ can be taken to be zero.
It remains to study $B(\chi)$. Let $\sigma=\ord_p(\chi(\gamma)-1)=\frac{1}{p^{m-1}(p-1)}<1$. We have
\begin{align*}
    \ord_p\left(\frac{(\chi(\gamma)-1)^n}{n}\right) = n\sigma - \ord_p(n) \ge \frac{1}{p-1} - m + 1,
\end{align*}
where the equality is attained at $n=p^m$. This tells us that $B(\chi)=\frac{1}{p-1}+1-m$, which concludes the proof.
\end{proof}

\begin{comment}
\begin{remark}
    In the next section, we shall work with the family of height-one primes $\{\fP_m\}_{m\ge 1}$. The hypothesis $0<\ord_p(\chi(\gamma)-1)<1$ is satisfied as long as $m\ge 2$. \end{remark}
\end{comment}

In order to define the remaining constants $\beta_\fP$ and $b_\fP'$, we need to introduce a number of notions.
For the remainder  of the current section, we let $\fP$ denote  a height-one prime of the form  $\fP_m=(\Theta_m)$  of $\Lambda$  for some $m$ and write $\chi$ for the corresponding character $\chi_\fP$. We will also fix a uniformizer $\varpi_{\fP}$ in $S_{\fP}$. 

\begin{definition}\label{def:PR-stuff}
    \item[i)] Let $H$ denote the Hilbert class field of $K$ and write $\hat{H}$ for the completion of $H$ at  the prime above $p$ determined by our chosen embedding $\iota_p:\overline{\QQ}\hookrightarrow\C_p$. The ring of integers of $\hat{H}$ is denoted by $\sW$.
    \item[ii)]We write $\cG_{p^\infty}$ for the Galois group of the ring class field of $K$ of conductor $p^\infty$. The algebra of $\fF$-valued distributions on $\cG_{p^\infty}$ is denoted by $\sH_\infty(\cG_{p^\infty})$.
    \item[iii)]We define $\HIw(\hat H, T_f)$ to be the inverse limit   $    \varprojlim H^1(\hat{H}_{p^n},T_f)$,
    where $\hat{H}_{p^n}$ denotes the completion of the ring class field of $K$ of conductor $p^n$ at the prime above $\fp$ determined by $\iota_p$ and the connecting maps are corestrictions. 
    \item[iv)]Let $\hat\cG_{p^\infty}$ denote the Galois group of $\bigcup_{n\ge0}\hat{H}_{p^n}$ over $K_\fp$. The Iwasawa algebra $\cO[[\hat\cG_{p^\infty}]]$ is denoted by $\hat\Lambda$.
    \item[v)]We let $\tilde{\Omega}_{V_f,1}^\epsilon:\sW[[\cG_{p^\infty}]]\otimes_{\Zp} \Dcris(V_f|_{G_{K_\p}})\rightarrow \HIw(\hat H,T_f)\otimes_{\hat\Lambda} \sH_\infty(\cG_{p^\infty})$ denote the semi-local Perrin-Riou exponential map defined in \cite[\S10]{Kob}\footnote{For simplicity, we have identified $\tilde\fR_1^{\psi=0}$ with $\sW[[\cG_{p^\infty}]]$ by Proposition 5.15 \textit{loc. cit.}}.
    \item[vi)]Let  $M_\alpha$  be the lattice inside $\Dcris(W_f)^{\vp=\alpha}(1)=\Dcris(V_f)^{\vp=\alpha/p}$  such that $$[ M_\alpha,\Fil^0\Dcris(T_f)]_V=\cO,$$ where $[-,-]_V$ is the natural pairing on $\Dcris(V)$ (see \cite[proof of Proposition~4.15]{KobOta} and \cite[\S7.2]{Kob}). 
    \item[vii)]Let \[
\Omega_{\fP,\alpha}:\sW[[\cG_{p^\infty}]]\otimes_{\Zp}  M_\alpha\rightarrow H^1(K_\fp, V_\fP)
\]
denote the restriction of the $\chi$-specialization of  $\tilde{\Omega}_{V_f,1}^\epsilon$ given as in \cite[\S4.4]{KobOta}. 
\end{definition}

Since $\alpha$ will be fixed throughout, we often drop it from the subscript when there is no confusion.

\begin{lemma}\label{lem:projection}
    Let $\pr_{\chi}:H^1(K_\fp,\TT)\rightarrow H^1(K_\fp,\TT(\chi))= H^1(K_\fp,T_\fP)$ be the natural projection map (see \cite[\S3.1]{KobOta} where we have specialized to the case $n=0$). Then $\pr_\chi$ is surjective.
\end{lemma}
\begin{proof}
 By duality, it is enough to prove the restriction map
    \[
H^1(K_\fp,A(\chi^{-1}))\rightarrow H^1(K_{\infty,\fp},A(\chi^{-1}))
    \]
    is injective.  Indeed, this kernel is given by
\[H^1\left(\Gamma,H^0(K_{\infty,\fp},A(\chi^{-1}))\right).\]
As $\chi$ factors through $\Gamma$, we have $H^0(K_{\infty,\fp},A(\chi^{-1}))=H^0(K_{\infty,\fp},A)\otimes S_\fP$. Thus, it suffices to show that $H^0(K_{\infty,\fp},A[\varpi])=0$, where $A[\varpi]$ is the $\varpi$-torsion of $A$.

Let $I$ be the inertia group of $G_{K_\fp}$. Since $f$ is non-ordinary at $p$, \cite[Theorem~2.6]{edixhoven} tells us that (using the duality $A[\varpi]\simeq \Hom(T/\varpi,\mu_p)$)
$$A[\varpi]|_I=\begin{pmatrix}
    \psi&0\\
    0&\psi'
\end{pmatrix},$$ 
where $\psi$ and $\psi'$ are fundamental characters of level $2$. Hence, the same proof as \cite[Lemma~4.4, case 1]{lei09} implies that $H^0(K_{\infty,\fp},A[\varpi])=0$, as desired.  
\end{proof}

Recall $M=M_\alpha$ from Definition~\ref{def:PR-stuff}vi).
The map $\Omega_{\fP,\alpha}$ can be realized as
\[
\sW[[\cG_{p^\infty}]]\otimes_{\Zp}  M\rightarrow
\Lambda\otimes M\rightarrow H^1(K_\fp,\TT)\otimes \sH_\infty(\Gamma) \rightarrow H^1(K_\fp, V_\fP),
\]
where $\sH_\infty(\Gamma)$ is the algebra of $\fF$-valued distributions on $\Gamma$, the first arrow is the natural norm map, the second arrow is the $\Lambda$-linear Perrin-Riou map, which we shall denote by $\Omega_{V_f,1}^\epsilon$, and the last arrow is given by $\chi_\fP\otimes \pr_\chi$, where $\pr_\chi$ is defined as in the statement of Lemma~\ref{lem:projection}.

Let ${H}^1_{\mathscr{F}_{\mathfrak{P}}}(K_\fp,T_{\mathfrak{P}})\subset H^1(K_\fp,T_{\mathfrak{P}})$ be the subgroup given by the Selmer structure $\mathscr{F}_\fP$ defined in \cite[right after Lemma~4.6]{KobOta}. It is a torsion-free $S_\fP$-module since 
$$
H^0(K_\fp,A_{\fP})\subset H^0(K_{\infty,\fp},A)\otimes S_\fP=0
$$ by the proof of Lemma~\ref{lem:projection}. 
In fact, it is free of rank-one over $S_{\mathfrak{P}}$ as discussed in \cite[paragraph preceding Lemma 4.9]{KobOta}.

 As in \cite[Lemma 4.9]{KobOta}, we define $\beta_\fP$ to be a constant which ensures that
 \[
\beta_\fP\Omega_\fP(\tilde{\mathfrak{R}}_1^{\psi=0}\otimes M)\subset {H}^1_{\mathscr{F}_{\mathfrak{P}}}(K_\fp,T_{\mathfrak{P}}).
 \]
We put
	\begin{align*}
		b'_{\mathfrak{P}} = \mathrm{length}_{S_{\mathfrak{P}}}({H}^1_{\mathscr{F}_{\mathfrak{P}}}(K_\fp,T_{\mathfrak{P}})/\beta_{\mathfrak{P}}\Omega_{\mathfrak{P}}(\tilde{\mathfrak{R}}_1^{\psi=0}\otimes M))
	\end{align*}
	as in the proof of Proposition 4.15, \textit{ibid}. Letting $\varpi_{\mathfrak{P}}$ be a uniformizer of $S_{\mathfrak{P}}$, we then have
	\begin{align}\label{definition:b_P}
		b_{\mathfrak{P}} = \varpi_{\mathfrak{P}}^{b'_{\mathfrak{P}}} q_{\mathfrak{P}} \beta_{\mathfrak{P}}^{-1}.
	\end{align}
	(Note that instead of $\varpi_{\mathfrak{P}}^{b'_{\mathfrak{P}}}$, Kobayashi--Ota put $p^{b'_{\mathfrak{P}}}$ in the definition of $b_{\mathfrak{P}}$, which seems to be a typo, as the image of $\beta_{\mathfrak{P}}\Omega_{\mathfrak{P}}(\boldsymbol{h}_z)$ only coincides with $\varpi_{\mathfrak{P}}^{b'_{\mathfrak{P}}}z$ instead of $p^{b'_{\mathfrak{P}}}z$ \textit{loc.~cit}.)

 We conclude this section with the following proposition regarding the valuation of this constant.

\begin{proposition}\label{prop:b_P}
    Let $m$ be an even integer. For $\fP=(\Theta_m)$, we have
    \[
    \ord_p(b_\fP)\le o(1),
    \]
    where $o(1)$ denotes a quantity that tends to 0 as $m$ tends to infinity.
\end{proposition}
\begin{proof}
Let $e_\alpha$ be the $\vp$-eigenvector of $\Dcris(V)$ given in Corollary~\ref{cor:explicit-matrices}. It can be checked that $M_\alpha$ is generated by $\frac{1}{\alpha} e_\alpha$. It follows from Corollary \ref{cor:explicit-matrices} and Lemma \ref{lem:projection} that
\begin{align*}
    \Omega_{\fP,\alpha}(\tilde{\fR}^{\psi=0}_1\otimes M_\alpha) = \alpha^{-m-1}p^{\frac{p}{p^2-1}+o(1)}H^1_{\sF_\fP}(K_\fp,T_{\fP}),
\end{align*}
whereby we have
\begin{align*}
    b_\fP' &= \mathrm{length}_{S_\fP}\left({H}^1_{\sF_\fP}(K_\fp,T_\fP)/\beta_\fP\Omega_{\fP,\alpha}(\tilde{\fR}^{\psi=0}_1\otimes M_\alpha)\right)\\
    &=\ord_{\varpi_\fP}(\beta_\fP)-(m+1)\ord_{\varpi_\fP}(\alpha)+\left(\frac{p}{p^2-1}+o(1)\right)\ord_{\varpi_\fP}(p).
\end{align*}
Since $b_\fp = \varpi_{\fP}^{b'_\fP}q_{\chi_\fP}\beta_\fP^{-1}$, using Lemma \ref{lem:q} we find
\begin{align*}
    \ord_p(b_{\fP}) &= \ord_{p}(\beta_\fP)-(m+1)\ord_{p}(\alpha)+\frac{p}{p^2-1}+o(1) \\
    &\hspace{1.5cm}+ \ord_p(\alpha)\left(m-\frac{1}{p-1}\right) - \ord_p(\beta_\fP)\\
    &=\frac{p}{p^2-1}-\frac{p}{p-1}\ord_p(\alpha)+o(1) .
\end{align*}
Our running hypothesis that  $\ord_p(a_p)\ge\frac{1}{p+1}$  implies $\ord_p(\alpha)\ge \frac{1}{p+1}$. Hence the proposition follows.
\end{proof}

\section{The proof of Theorem~\ref{thmA}}\label{sec:proof}
In this section, we adopt the method employed by Howard  in \cite[proof of Theorem~2.2.10]{howard} in our setting to prove Theorem~\ref{thmA}. Recall that $\mathfrak{P}_m$ denotes the height-one prime generated by $\Theta_m(X)=\prod_{\sigma\in\Gal(\fF(\mu_{p^m})/\fF)}(X+1-\zeta_{p^m}^\sigma)$. Below we write $\Theta_m(X) = X^{k(m)} + t_{k(m)-1}(m)X^{k(m)-1} +\cdots + t_0(m)$, and will often omit $m$ from the notation when there is no ambiguity. We begin with the following preliminary lemmas:

\begin{lemma}
	\label{lem:length-pi-multiple}
	Let $R$ be a finite free $\mathcal{O}$-module and $R_0\subseteq R$ an $\mathcal{O}$-submodule of finite index. Then for any $t\in \mathbf{Z}_{\ge 0}$,
	\begin{align*}
		{\rm length}_{\mathcal{O}}(R/\varpi^t R_0) = {\rm length}_{\mathcal{O}}(R/R_0) + t\cdot \rank_{\mathcal{O}}(R). 	\end{align*} 
\end{lemma}

\begin{proof}
	This follows from the theory of elementary divisors.
\end{proof}

\begin{lemma}\label{lem:t0-divisible}
	 For $m$ sufficiently large, $\ord_p(t_0(m))\le \ord_p(t_i(m))$ for all $0\le i<k$.
\end{lemma}
\begin{proof}
	Let $F$ be the  subfield of $\Q(\mu_{p^m})$ fixed by $\Gal(\fF(\mu_{p^m})/\fF)\subseteq \Gal(\Q(\mu_{p^m})/\Q)$. Then $\Theta_m(X)\in F[X]$. As such, by computing the $p$-valuation we see that $t_0 = -\mathrm{Nm}_{\Q(\mu_{p^m})/F}(\zeta_{p^m}-1)$ is a uniformizer of $F$. Since all $t_i$'s have positive valuations, the result follows.
\end{proof}

\begin{lemma}
	\label{lem:chinese-remainder-fake}
	Let $M$ be a finitely generated $\Lambda$-torsion module with $M \simeq \Lambda/(\varpi^\mu P)$, where $P\in \Lambda$ is a distinguished polynomial, \textit{i.e.}, $P(X) = X^{\lambda} + a_{\lambda-1}X^{\lambda-1} + \cdots + a_0$ with $|a_i|_p<1$ for $0\le i<\lambda$. The cardinality of the kernel of the natural projection
	\begin{align*}
		M\otimes_{\Lambda} \Lambda/\mathfrak{P}_m \to \Lambda/(\varpi^\mu,\mathfrak{P}_m)
	\end{align*}
	is bounded independently of $m$.
\end{lemma}

\begin{proof}
	Under the identification $M\simeq \Lambda/(\varpi^\mu P)$, we see that the kernel corresponds to $(\varpi^\mu,\mathfrak{P}_m)/(\varpi^\mu P, \mathfrak{P}_m)$. In the proof of Lemma \ref{lem:bound-torsion}, we have seen that the constant term of $\Theta_m$, $t_0$, is independent of $m$; write $e=\ord_{\varpi}(t_0)$. First, we show that the $\Lambda$-submodule
	\begin{align*}
		\fN: = (\varpi^{\mu+e},\varpi^\mu P, \mathfrak{P}_m)/(\varpi^\mu P, \mathfrak{P}_m) 
	\end{align*}
	is zero for large $m$. By Nakayama's lemma, it suffices to show that $\fN \subseteq (\varpi,X)\fN$, which boils down to the relation $\varpi^{\mu+e} \in (\varpi, X)\fN$.  Lemma \ref{lem:t0-divisible} implies that 
	\begin{align*}
		U(X) = -\frac{t_0+t_1X+\cdots + t_{k-1}X^{k-1}}{\varpi^e}\in \Lambda^\times,
	\end{align*}
	and we thus have $\varpi^e \equiv X^{k}U^{-1} \bmod \fP_m$ and $\varpi^{\mu+e} \equiv \varpi^\mu U^{-1}X^k \bmod\mathfrak{P}_m$.
	We may write $k = l\lambda + r$ for some $l,r\in \Z_{\ge 0}$ and $0\le r<\lambda$. Then 
	\begin{align*}
		\varpi^\mu X^{k} = X^r\varpi^\mu(X^\lambda)^l \equiv X^r\varpi^\mu(-1)^l(a_{\lambda-1}X^{\lambda-1}+\cdots+a_0)^l \pmod{\varpi^\mu P}.
	\end{align*}
	Since $l>(k(m)-\lambda)/\lambda = |\Gal(\fF(\mu_{p^m})/\fF)|/\lambda -1$, we see that the last term is  divisible by a large power of $\varpi$ when $m$ is large enough, which implies that $\varpi^{\mu+e} \in (\varpi, X)\fN$ and thus proving our claim.
	
	\vspace{3mm}
	
	Knowing that $\fN=0$, it is now enough to bound $(\varpi^\mu, \mathfrak{P}_m)/(\varpi^{\mu+e},\varpi^\mu P, \mathfrak{P}_m)$, which in turn has cardinality no larger than $(\varpi^\mu)/(\varpi^{\mu+e},\varpi^\mu P)\simeq \Lambda/(\varpi^e,P)$. The latter is a finite $\Lambda$-torsion module that is independent of $m$, which concludes the proof.
\end{proof}

We now turn to the proof of Theorem \ref{thmA}. Recall from \cite[Theorem 3.9]{CastellaHsieh} that $\sL\ne 0$. We will let $\mathfrak{P}$ range in the family $\{\mathfrak{P}_m\}_{m\ge m_0}$, where $m_0$ is a large enough even integer such that for all $m\ge m_0$, $\chi_{\mathfrak{P}_m}(\sL)\ne 0$ and $\mathfrak{P}_m$ is not in the finite set $\Sigma$ of Proposition \ref{prop.control-theorem}. By Proposition \ref{prop:KO-length}, we have an inequality
\begin{equation}\label{ineq}
	2 \cdot {\rm length}_{\mathcal O}\left(\frac{S_{\mathfrak{P}_m}}{ b_{\mathfrak{P}_m}\chi_{\mathfrak{P}_m}(\mathscr{L})}\right) \ge {\rm length }_{\mathcal O}(X_{(0,\emptyset)}(K,A_{\mathfrak{P}_m})).
\end{equation}

For the left-hand side, applying Lemma \ref{lem:length-pi-multiple} to the $\mathcal{O}$-module  $R = S_{\mathfrak{P}_m}=\mathcal{O}[X]/(\Theta_m)$ (which is free of rank $k(m)=\deg(\Theta_m)$), and the submodule $R_0 = (\chi_{\mathfrak{P}_m}(\mathscr{L}))=\chi_{\fP}(\mathscr{L})S_{\mathfrak{P}_m}$, we deduce that
\[		{\rm length}_{\mathcal O}\left(\frac{S_{\mathfrak{P}_m}}{ b_{\mathfrak{P}_m}\chi_{\mathfrak{P}_m}(\mathscr{L})}\right)
= k(m) \cdot\ord_{\varpi}(b_{\mathfrak{P}_m})
+ {\rm length}_{\mathcal O}\left(S_{\mathfrak{P}_m}/\chi_{\mathfrak{P}_m}(\mathscr{L})\right).
\]
The second term on the right-hand side of this equation satisfies
\begin{align*}
	{\rm length}_{\mathcal O}\left(S_{\mathfrak{P}_m}/\chi_{\mathfrak{P}_m}(\mathscr{L})\right) &= {\rm length}_{\mathcal{O}}(\Lambda/(\mathfrak{P}_m, \mathscr{L}))\\
	&={\rm length}_{\mathcal{O}}(\Lambda/(\mathfrak{P}_m, \varpi^{\ord_{\varpi}(\mathscr{L})})) + O(1)\\
	&={\rm length}_{\mathcal{O}}(\cO/\varpi^{\ord_{\varpi}(\mathscr{L})}[X]/(\Theta_m)) + O(1)\\
	&=k(m)\cdot\ord_{\varpi}(\mathscr{L}) + O(1),
\end{align*}
where the second equality follows from Lemma \ref{lem:chinese-remainder-fake}. 

The length appearing on the right-hand side of \eqref{ineq} satisfies
\begin{align*}
	{\rm length }_{\mathcal O}(X_{(0,\emptyset)}(K,A_{\mathfrak{P}_m}))
	&= {\rm length }_{\mathcal O}(X_{(\emptyset,0)}(K,A_{\mathfrak{P}_m^\iota}))\\
	&= {\rm length}_{\mathcal{O}}(\mathcal{X}_{(\emptyset,0)}(f)/\mathfrak{P}_m^\iota) + O(1)\\
	&= {\rm length}_{\mathcal{O}}(\mathcal{X}_{(\emptyset,0)}(f)/\mathfrak{P}_m) + O(1)\\
	&= k(m)\cdot \ord_{\varpi}(\mathrm{Char}(\mathcal{X}_{(\emptyset,0)}(f)) + O(1),
\end{align*}
where the second equality follows from Proposition \ref{prop.control-theorem}. Combining these estimates, we deduce that for $m\ge m_0$, 
\begin{align*}
	2k(m)\cdot\ord_{\varpi}(b_{\mathfrak{P}_m}) + 2k(m)\cdot\ord_{\varpi}(\mathscr{L}) + O(1) \ge k(m)\cdot\ord_{\varpi}({\rm Char}(\mathcal{X}_{(\emptyset,0)}(f)) + O(1).
\end{align*}
Proposition \ref{prop:b_P} further tells us that for all large enough even integers $m$, we have $\ord_{\varpi}(b_{\fP_m})\le o(1)$. Thus,
\begin{align*}
	2k(m)\cdot\ord_{\varpi}(\mathscr{L}) + o(k(m)) \ge k(m)\cdot\ord_{\varpi}({\rm Char}(\mathcal{X}_{(\emptyset,0)}(f)) + O(1).
\end{align*}
If we divide both sides by $k(m)$ and let $m$  tend to $\infty$, we deduce that
\begin{align*}
	2\ord_{\varpi}(\mathscr{L}) \ge \ord_{\varpi}({\rm Char}(\mathcal{X}_{(\emptyset,0)}(f))).
\end{align*}
This concludes the proof of Theorem~\ref{thmA}.

 \appendix
 \section{Integrality of specializations of the Perrin-Riou map}
\label{appendix}

 \noindent \textbf{Notation note.} The notation and hypotheses in the main body of the article continue to be in force. For a $p$-adic Lie group $G$ and a $p$-adic ring $R$, we write $\Lambda_R(G)$ for the completed group ring of $G$ with coefficients in $R$, and $\sH_{\infty,R}(G)$ for the  algebra of $\Frac(R)$-valued distributions on $G$. If $R=\cO$, we often omit it from the notation.
	
	\subsection{Anticyclotomic Perrin-Riou regulator and Coleman maps}
	
	In this section, we review the construction of the anticyclotomic Perrin-Riou regulator and Coleman maps, by composing the two-variable counterparts with the projection map as done in \cite[§5.1]{CastellaHsieh}. We begin by recalling the two-variable Perrin-Riou map from \cite{loefflerzerbes14} and its decomposition given in \cite{BL2}. Let $\KK_\infty$ be the $\Z_p^2$-extension of $K$, and fix an embedding $\KK_\infty \to \C_p$ such that the induced place on $K$ is $\fp$; by an abuse of notation we will write this place of $\KK_\infty$ also as $\fp$. Since $p$ is assumed to be split in $K$, there is a natural identification $K_\fp\simeq \Q_p$, and the local completion $\KK_{\infty,\fp}$ can be identified with the compositum $F_\infty \Q_p^{\cyc}$, where $F_\infty$ is the unramified $\Z_p$-extension of $\Q_p$ and $\Q_p^\cyc$ is the cyclotomic $\Z_p$-extension of $\Q_p$. Put $U = \Gal(F_\infty/\Q_p)$ and $F_m = F_\infty^{U^{p^m}}$. When $m=0$, we shall write $F=F_0=\Qp$.

	The two-variable Perrin-Riou map constructed by Loeffler--Zerbes \cite{loefflerzerbes14} is given by the projective limit 
	\begin{align*}
		\cL_{T,F_\infty} = \varprojlim_m \cL_{T,F_m}: \varprojlim_m H^1_\Iw(F_m\Q_p^\cyc,T) 
		&\to \varprojlim_m \sH_{\infty}(\tilde{\Gamma}^{\cyc})\otimes_{\Q_p} \DD_{\cris}(T)\otimes_{\Z_p} \cO_{F_m}\\
		&\simeq \sH_{\infty}(\tilde{\Gamma}^{\cyc})\otimes_{\Q_p}\DD_\cris(T)\widehat{\otimes}_{\Z_p} S_{F_\infty/\Qp},
	\end{align*}
	where $\tilde{\Gamma}^{\cyc} = \Gal(F_\infty\Q_p^\cyc/F_\infty)\subset \Gal(F_\infty\Q_p^\cyc/K_\fp)$, $S_{F_\infty/\Qp}$ is the Yager module \cite[§3.2]{loefflerzerbes14}, and the transition maps in the inverse limits are given by corestrictions on the left, and trace maps on $\cO_{F_m}$ on the right \cite[Theorem 4.7]{loefflerzerbes14}. As a $\Lambda_{\Zp}(U)$-module, $S_{F_\infty/\Qp}$ is contained in $\Lambda_{\cO_{\hat{F}_\infty}}(U)$ and is free of rank one. We fix a basis $\cY$ of the $\Lambda_{\Zp}(U)$-module $S_{F_\infty/\Qp}$. 
	
	We fix an $\cO$-basis $\omega$ of $\mathrm{Fil}^0\mathbb{D}_{\cris}(T)$, then $v_1=\omega,v_2=\varphi(\omega)$ form an $\cO$-basis of $\mathbb{D}_{\cris}(T)$ (see, e.g., \cite[Lemma 3.1]{leiloefflerzerbes15}). Furthermore, for any intermediate extension $F_m$, we have Coleman maps
	\begin{align*}
		\Col_{\sharp/\flat,F_m}:H^1_{\Iw}(F_m\Q_p^\cyc,T)\to \Lambda(\Gal(F_m\Q_p^{\cyc}/F_m))\otimes_{\Z_p}\cO_{F_m},
	\end{align*}
	such that \cite[§2.3]{BL2}
	\begin{align*}
		\cL_{T,F_m} = (v_1,v_2)M_{\log}\begin{pmatrix}
			\Col_{\sharp,F_m}\\
			\Col_{\flat,F_m}\\
		\end{pmatrix},
	\end{align*}
	where $M_{\log}$ is the logarithm matrix, a $2\times 2$ matrix valued in $\sH_\infty(\Gal(F_m\Q_p^{\cyc}/F_m))\simeq \sH_\infty(\tilde{\Gamma}^{\cyc})$ that is independent of $m$. Taking the inverse limit, we thus obtain a factorization of the big logarithm:
	\begin{align}\label{eq:2-var-factorization}
		\cL_{T,F_\infty} = (v_1\ v_2)M_{\log} \begin{pmatrix}
			\Col_{\sharp,F_\infty}\\
			\Col_{\flat,F_\infty}\\
		\end{pmatrix},
	\end{align}
	where, for $\star\in \{\flat,\sharp\}$, 
	\begin{align*}
		\Col_{\star,F_\infty}:H^1_\Iw(\KK_{\infty,\fp},T)\to \Lambda(\tilde{\Gamma}^{\cyc})\widehat{\otimes} \cY\Lambda(U)\subset \Lambda_{\cO_{\hat{F}_\infty}}(\tilde{\Gamma}^{\cyc}\times U)
	\end{align*} 
	is the limit of $\Col_{\star,F_m}$.
	
	\begin{remark}\label{remark:image-basechange}
	    Since  $\cL_{T,F_\infty}$ is $\Lambda(U)$-linear \cite[paragraph after Definition 4.6]{loefflerzerbes14}, so are the Coleman maps defined above. As such, for $\star\in\{\sharp,\flat\}$, a consideration of the determinants of $\cL_{T,F_\infty}$ and $\Mlog$ shows that $\mathrm{im}(\Col_{\star,F_m})= \mathrm{im}(\Col_{\star,F})\otimes\cO_{F_m}\subseteq \Lambda(\Gal(F_m\Q_p^{\cyc}/F_m))\otimes_{\Z_p}\cO_{F_m}$. Thus
	    \begin{align*}
	        \mathrm{im}(\Col_{\star,F_\infty})= \mathrm{im}(\Col_{\star,F})\otimes\cY\Lambda(U)\subseteq \Lambda(\tilde{\Gamma}^{\cyc})\otimes\Lambda_{\cO_{\hat{F}_\infty}}(U).
	    \end{align*}
	\end{remark}
	
	We now project the factorization \eqref{eq:2-var-factorization} to the anticyclotomic subextension, following \cite[Theorem 5.1]{CastellaHsieh}. Let $K^{\rm ac}$ denote the anticyclotomic $\Z_p$-extension of $K$,\footnote{While it is denoted by $K_\infty$ in the main body of the article, here for clarity we use the notation $K^{\rm ac}$ instead.} and denote by $\Gamma^{\rm ac} = \Gal(K^{\rm ac}/K)\simeq \Gal(K^{\rm ac}_\fp/K_\fp)\simeq \Z_p$.
	\begin{theorem}
		By quotienting out $\mathfrak{J}=\ker(\Lambda(U\times \tilde{\Gamma}^{\cyc})\to \Lambda(\Gamma^{\rm ac}))$, the map $\cL_{T,F_\infty}$ descends to a $\Lambda(\Gamma^{\rm ac})$-linear map
		\begin{align*}
			\cL_{T}^{\ac}: H^1_\Iw(K^{\rm ac}_\fp, T) \to \sH_{\infty,\hat{F}_\infty}(\Gamma^{\rm ac})\otimes \DD_\cris(T).
		\end{align*}
		Therefore, there exist Coleman maps $\Col_{\sharp}^{\rm ac},\Col_{\flat}^{\rm ac}: H^1_\Iw(K^{\rm ac}_\fp,T)\to \Lambda(\Gamma^{\rm ac})$, being the projections of $\Col_{\sharp,F_\infty},\Col_{\flat,F_\infty}$ in $\Lambda(\Gamma^{\rm ac})$, such that the following factorization holds for all $z\in H^1_\Iw(K^{\rm ac}_\fp,T)$
		\begin{align*}
			\cL_{T}^{\ac}(z) = (v_1\ v_2)\overline{M}_{\log} \begin{pmatrix}
				\Col_{\sharp}^{\ac}(z)\\
				\Col_{\flat}^{\ac}(z)\\
			\end{pmatrix}.
		\end{align*}
		Here, $\overline{M}_{\log}$ denotes the image of $\Mlog$ in $M_2(\sH_{k-1,\fF}(\tilde{\Gamma}^\cyc)/\mathfrak{J})=M_2(\sH_{k-1,\fF}(\Gamma^{\rm ac}))$.
	\end{theorem}
	\begin{proof}
    Recall from \cite[Lemma 2.7]{KobOta} that $V^{\Gal(K^{\rm ac}_\fp/K)}=0$. Thus, the theorem follows from \cite[Theorem 5.1]{CastellaHsieh}.
	\end{proof}
	\begin{remark}\label{rmk:relative-position}
		As explained in \cite[Remark 5.2]{leisprung}, we may identify $\Gal(k_{\infty,\fp}/K_\fp)$ with $\Z_p^2$ such that $U$ is the first component and $\tilde{\Gamma}^{\cyc}$ is the second component. As such, $\Gamma^{\rm ac}$ corresponds to anti-diagonal elements
		\begin{align*}
			\{(a,-a)\in \Z_p^2: a\in \Z_p\}.
		\end{align*}
	\end{remark}

	\subsection{Evaluating the logarithm matrix}
	
	In what follows, we fix a family of primitive $p^n$-th roots of unity $\zeta_{p^n}$ and write $\epsilon_n = \zeta_{p^n}-1$. We may choose topological generators $\gamma^{\ur}\in U$ and $\gamma^\cyc\in \tilde{\Gamma}^\cyc$ resulting in a topological generator $\gamma$ of $ \Gamma^{\rm ac}$ given by
	\begin{align*}
		\gamma = (\gamma^\cyc/\gamma^\ur)^{1/2},
	\end{align*}
	which is possible by Remark \ref{rmk:relative-position} and that $p\ne 2$. By these choices, for $G\in \{U,\tilde{\Gamma}^{\cyc},\Gamma^{\rm ac}\}$, we regard $\sH_{\infty}(G)$ as the set of power series convergent on the open unit disc centered at 0 with variable $X_G$. To simplify notation, we write $Y=X_U$, $Z = X_{\tilde{\Gamma}^\cyc}$ and $X=X_{\Gamma^\ac}$. As $\gamma = (\gamma^\cyc/\gamma^\ur)^{1/2}$, the natural projection
	\begin{align*}
		\sH_\infty(U\times \tilde{\Gamma}^\cyc) = \sH_\infty(U)\hat{\otimes}\sH_\infty(\tilde{\Gamma}^\cyc) \to \sH_\infty(\Gamma^{\rm ac})
	\end{align*}
	is given by sending $f(Y,Z)$ to $f((1+X)^{-1}-1,X)$. 
	
	We now turn to the explicit description of $M_{\log}$ and thus $\overline{M}_{\log}$, following \cite[§2]{BL-mathZ}. Denote by $\Phi_{p^n}(Z)$ the $p^n$-th cyclotomic polynomial $\frac{(Z+1)^{p^n}-1}{(Z+1)^{p^{n-1}}-1}$.
	Additionally, recall the matrices
	\begin{align*}
		A=
		\begin{pmatrix}
			0 & -1/p\\
			1 & a_p/p
		\end{pmatrix}\quad and \quad
		Q_n(Z) =
		\begin{pmatrix}
			a_p & 1\\
			-\Phi_{p^{n+1}}(Z) & 0\\
		\end{pmatrix}\ (n\ge 0).
	\end{align*} 
	Proposition~2.5 \textit{ibid.} tells us that
	\begin{align*}
		M_{\log}(Y,Z) = M_{\log}(Z) = \lim_{n\to \infty} A(Z)^{n+1}Q_{n-1}(Z)\cdots Q_0(Z).
	\end{align*}
	Consequently,
	\begin{align*}
		\overline{M}_{\log}(X) = M_{\log}(X) = \lim_{n\to \infty} A(X)^{n+1}Q_{n-1}(X)\cdots Q_0(X).
	\end{align*}
	Henceforward we shall not distinguish $\overline{M}_{\log}$ from $M_{\log}$.
	
	Next, given $\phi=\begin{pmatrix}
		a & b \\ c & d\\
	\end{pmatrix}$ with entries valued in $\overline{\Q_p}$, following the notation introduced in \cite[Definition~4.4]{sprung13},  we write
	\begin{align*}
		\ord_p(\phi) = \begin{pmatrix}
			\ord_p(a) & \ord_p(b)\\
			\ord_p(c) & \ord_p(d)\\
		\end{pmatrix}.
	\end{align*}
	Further, we denote $\ord_p(a_p)$ by $v$. To state the result below, we recall from  \cite[Theorem 2.1]{colemanedixhoven} that the two roots of $X^2 - a_p X + p$ are distinct since $f$ is of weight 2.
	\begin{lemma}\label{lem:log-matrix}
		 Let $\alpha\ne \beta$ be the two roots of the Hecke polynomial $X^2 - a_p X + p$. Also let $S$ denote the matrix
		\begin{align*}
		    S = \begin{pmatrix}
		        1 & 1\\
		        -\alpha & -\beta
		    \end{pmatrix}.
		\end{align*}
		Then $M_{\log}(\epsilon_n)$ is of the form
		\begin{align*}(\alpha-\beta)^{-1}S
		    \begin{pmatrix}
		        -s_1/\beta^n & -s_2/\beta^n\\
		        s_1/\alpha^n & s_2/\alpha^n
		    \end{pmatrix},
		\end{align*}
		for some $s_1,s_2\in \overline{\Qp}$ of $p$-adic valuations $v\boldsymbol{1}_{2\nmid n}+\sum_{1\le i\le n/2}p^{-2i+1}$ and $v\boldsymbol{1}_{2\mid n} + \sum_{1\le i< n/2}p^{-2i}$ respectively.
	\end{lemma}
	\begin{proof}
		Note that for $i\ge n$, we have $\Phi_{p^{i+1}}(\epsilon_n) = p$,  which implies that $A = Q_i(\epsilon_n)^{-1}$. This implies that $M_{\log}(\epsilon_n) = A^{n+1}Q_{n-1}\cdots Q_0(\epsilon_n)$. By \cite[Proposition 4.6]{leiloefflerzerbes15},  we have
		\begin{align*}
			\ord_p(Q_{n-1}\cdots Q_0(\epsilon_n))=
			\begin{cases}
				\begin{pmatrix}
		    	v+\sum_{i=1}^{\frac{n-1}{2}} \frac{1}{p^{2i-1}} & \sum_{i=1}^{\frac{n-1}{2}} \frac{1}{p^{2i}}\\
					\infty & \infty\\
				\end{pmatrix}
				& \text{if }n\text{ is odd,}\\
				\begin{pmatrix}
					\sum_{i=1}^{\frac{n}{2}} \frac{1}{p^{2i-1}} & v+\sum_{i=1}^{\frac{n}{2}-1} \frac{1}{p^{2i}}\\
					\infty & \infty\\
				\end{pmatrix}
				& \text{if }n\text{ is even}.\\
			\end{cases}
		\end{align*}
		
		For the matrix $A$, we have the diagonalization
		\begin{align*}
			A = p^{-1}S\begin{pmatrix}
				\alpha & 0\\ 0 & \beta \\
			\end{pmatrix}S^{-1}.
		\end{align*}
		Write $Q_{n-1}\cdots Q_0(\epsilon_n) = \begin{pmatrix}
			s_1 & s_2 \\ 0 & 0\\
		\end{pmatrix}$, we have
		\begin{align*}
			S^{-1}A^{n+1}Q_{n-1}\cdots Q_0(\epsilon_n)
			&= p^{-n-1}
				\begin{pmatrix}
					\alpha^{n+1} & 0 \\ 0 & \beta^{n+1}
				\end{pmatrix} S^{-1}
				\begin{pmatrix}
					s_1 & s_2 \\ 0 & 0\\
				\end{pmatrix}\\
			&=p^{-n-1}
				\begin{pmatrix}
					\alpha^{n+1} & 0 \\ 0 & \beta^{n+1}
				\end{pmatrix}
				\begin{pmatrix}
					-\beta s_1 & -\beta s_2\\
					\alpha s_1 & \alpha s_2
				\end{pmatrix}
				\frac{1}{(\alpha-\beta)}\\
			&=\frac{p^{-n-1}}{\alpha-\beta}
				\begin{pmatrix}
					-p\alpha^n & 0\\
					0 & p\beta^n\\
				\end{pmatrix}
				\begin{pmatrix}
					s_1  & s_2\\
					s_1 & s_2\\
				\end{pmatrix}\\
			&=\frac{1}{\alpha-\beta}
			\begin{pmatrix}
					-s_1/\beta^n  & -s_2/\beta^n\\
					s_1/\alpha^n & s_2/\alpha^n\\
			\end{pmatrix}.
		\end{align*}
	\end{proof}

	\subsection{Evaluation of Coleman maps}
	
	We shall evaluate the images of the Coleman maps at $\epsilon_n$ using \cite[Proposition 2.2]{LL-IJNT}. Write $\overline{\cY}$ as the anticyclotomic projection of $\cY$, and we define $\underline{\Col} = (\Col_\sharp^\ac,\Col_\flat^\ac): \HIw(K_\fp^\ac,T)\to \Lambda(\Gamma^{\ac})^{\oplus 2}$.
	\begin{proposition}
		Let
		\begin{align*}
		    I_v = \{(G_1,G_2)\in \overline{\cY}\Lambda(\Gamma^{\ac})^{\oplus 2}\subseteq \Lambda_{\cO_{\hat{F}_\infty}}(\Gamma^{\ac})^{\oplus 2}: (p-1)G_1(0) = (2-a_p)G_2(0)\}.
		\end{align*}
		Then $\mathrm{im}(\underline{\Col})=I_v$.
	\end{proposition}
	\begin{proof}
		It follows from \cite[Proposition 2.2]{LL-IJNT} that 
		\begin{align*}
		    \mathrm{im}(\Col_{\sharp,F},\Col_{\flat,F}) = \{(G_1,G_2)\in \Lambda(\Gal(\Q_p^{\cyc}/\Q_p))^{\oplus 2}: (p-1)G_1(0) = (2-a_p)G_2(0)\}.
		\end{align*}
		Thus, the affirmation on $\mathrm{im}(\underline{\Col})$ follows from Remark \ref{remark:image-basechange}.
	\end{proof}
	
	\begin{lemma}\label{lem:period-invertible}
    The period $\cY\in \Lambda_{\cO_{\hat{F}_\infty}}(U)$ is invertible, and thus so is $\overline{\cY}\in \Lambda_{\cO_{\hat{F}_\infty}}(\Gamma^{\ac})$.
	\end{lemma}
	\begin{proof}
	    The period $\cY$ is constructed from choosing a compatible system of integral normal basis generator $(x_{F_m})_{m\ge 0}\in \varprojlim_m \cO_{F_m}$, which is identified with an element of $\Lambda_{\cO_{\hat{F}_\infty}}(U)$ via the maps 
	    \begin{align*}
	        y_{F_m/F}: \cO_{F_m} \to \cO_{F_m}[\Gal(F_m/F)],\quad x\mapsto \sum_{\sigma\in\Gal(F_m/F)} x^{\sigma}[\sigma^{-1}],
	    \end{align*} (see \cite[\S3.2]{loefflerzerbes14}). 
	    Thus, the constant term of $\cY$ as a power series is $\lim_m \Tr_{F_m/F}(x_{F_m})=x_F$, which is a unit of $\cO_F$ since $\cO_F \cdot x_F = \cO_F$, from which  the lemma follows.
	\end{proof}
	Henceforth, we shall use the same notation for $\cY$ and $\overline{\cY}$ for presentational simplicity.
	\begin{corollary}\label{cor:image-basis-col}
	    There exists a $\Lambda(\Gamma^\ac)$-basis $(z_1,z_2)$ of $H^1_{\Iw}(K_\fp,T)$ such that
	    \begin{align*}
	        \underline{\Col}(z_1) = \cY(X,0), \  \underline{\Col}(z_2) = \cY(a',1),
	    \end{align*}
	    where $a'=\frac{2-a_p}{p-1}$.
	\end{corollary}
	\begin{proof}
	    It can be checked that $X\oplus 0,$ and $a'\oplus 1$ form a $\Lambda$-basis of the image of $\cY^{-1}\uCol$. Thus, the result follows from the injectivity of $\uCol$ (see \cite[Proof of Corollary 4.6]{BL2}).
	\end{proof}
	
	\vspace{3mm}
	
	Next we compare the maps $\cL_{T}^\ac$ and $\tilde{\Omega}^{\epsilon}_{V,1}$ constructed by Kobayashi \cite[\S10]{Kob}, using the explicit reciprocity law.

	\begin{lemma}\label{lem:exp-log-composition}
	There exists a unit in $u^\epsilon_T\in\Lambda_{\cO_{\hat F_\infty}}(\Gamma^\ac)$ such that 
	\[
	\cL^\ac_T\circ \tilde{\Omega}^{\epsilon}_{V,1}=  u^\epsilon_T\ell_0,
	\]
	where $\ell_0=\log(\gamma)/\log(\kappa(\gamma))$ and $\kappa$ is the Lubin--Tate character attached to the extension $K^\ac_\p/\Qp$.
	\end{lemma}
	\begin{proof}
	Let $\Col^\epsilon_e:\HIw(K_\p^\ac,T)\rightarrow\sH_\infty(\Gamma^\ac)\otimes \Dcris(T)$ be the map defined in \cite[Definition~3.3]{CH22}, where $e$ is some fixed unit in $\Lambda(\Gamma^{\ac})$ (note that we are taking $F=\Qp$ in \textit{loc. cit.}). Upon comparing the interpolation formulae given in Theorem~3.4 of \textit{op. cit.} and \cite[Theorem~5.1]{CastellaHsieh}, we see that $\cL_T^\ac$ and $\Col^\epsilon_e$ agree up to a unit. Therefore, it is enough to study the composition $\Col^\epsilon_e$ and $\tilde\Omega_{V,1}^\epsilon$.
	
	Let $[-,-]_V$ and $\langle-,-\rangle_{F_\infty}$ be the pairings defined in \cite[\S3.3]{CH22}. Then (3,7) in \textit{op. cit.} says that
	\[
	[\Col_e^\epsilon (z),\eta]_V=\langle z,\Omega_{V,1}^\epsilon (\eta\otimes e)\rangle_{F_\infty}
	\]
	for all $z\in \HIw(K_\p^\ac,T)$ and $\eta\in \Dcris(T)$. Thus, given any $x\in \sH_\infty(\Gamma^\ac)\otimes \Dcris(T)$, we deduce from the explicit reciprocity law (see \cite[Theorem~10.13]{Kob}) 
	\begin{align*}
	    [\Col_e^\epsilon \circ \Omega_{V,1}^\epsilon(x),\eta]_V&=\langle \Omega_{V,1}^\epsilon(x),\Omega_{V,1}^\epsilon (\eta\otimes e)\rangle_{F_\infty}\\
	    &=\langle\ell_0 \Omega_{V,0}^\epsilon(x),\Omega_{V,1}^\epsilon (\eta\otimes e)\rangle_{F_\infty}\\
	    &=[ \ell_0 x,\eta\otimes e]_V\\
	    &=[ \ell_0e^\iota x,\eta]_V.
	\end{align*}
	(Note that the image of $\delta_{-1}$ in \textit{loc. cit.} is sent to the trivial element in $\Gamma^\ac$.) Therefore, the result follows from the non-degeneracy of the pairing $[-,-]_V$.
	\end{proof}
	
	\begin{corollary}
	    Let $e_{\alpha},e_{\beta}$ be a $\varphi$-eigenbasis of $\DD_{\cris}(V)$ given by
	    \begin{align*}
	        (e_\alpha,e_\beta) = (\omega,\varphi(\omega))S = (\omega,\varphi(\omega)) \begin{pmatrix}
	            1 & 1\\ -\alpha & -\beta\\
	        \end{pmatrix}
	    \end{align*}
	    (with $\varphi(e_{\lambda})=\lambda p^{-1}e_{\lambda}$ for $\lambda\in\{\alpha,\beta\}$). The matrix of $\cL_{T}^\ac$ with respect to the bases $(z_1,z_2)$ and $(e_{\alpha},e_{\beta})$ is given by
    \[\cY S^{-1}
\Mlog\cdot\begin{pmatrix}
    X&a'\\0&1
\end{pmatrix}.
    \]
    The matrix of $\tilde{\Omega}^{\epsilon}_{V,1}$ with respect to the same bases is given by 
    \begin{align*}
        \frac{u^\epsilon_T \ell_0}{\cY X}\begin{pmatrix}
            1 & -a' \\ 0 & X\\
        \end{pmatrix}
        M_{\log}^{-1}S.
    \end{align*}
	\end{corollary}
    \begin{proof}
        By Corollary \ref{cor:image-basis-col}, we have a basis $z_1,z_2$ of $H^1_\Iw(K_\fp,T)$ such that
        \begin{align*}
            (\cL_T^\ac(z_1),\cL_T^\ac(z_2)) = (\omega, \varphi(\omega))M_{\log}\cY\begin{pmatrix}
                X & a'\\
                0 & 1\\
            \end{pmatrix}.
        \end{align*}
        The affirmation regarding $\cL_T^{\ac}$ now follows from the change of variable formula in the definition of $e_\alpha,e_\beta$. Taking the inverse of $\cL^\ac_T$ in Lemma \ref{lem:exp-log-composition} gives the matrix for $\tilde{\Omega}^\epsilon_{V,1}$.
    \end{proof}
    
    \begin{corollary}\label{cor:explicit-matrices}
        For large enough even integers $m$, the specialization of $\tilde{\Omega}^\epsilon_{V,1}$ at $\fP_m=(\Theta_m(X))$ has $p$-adic valuation matrix
        \begin{align*}
            \begin{pmatrix}
                -m\cdot\ord_p(\alpha) + \displaystyle\sum_{i=1}^{m/2}\frac{1}{p^{2i-1}} &  \displaystyle-m\cdot\ord_p(\beta) + \sum_{i=1}^{m/2}\frac{1}{p^{2i-1}} \\
                 \displaystyle-m\cdot\ord_p(\alpha) + \sum_{i=1}^{ m/2}\frac{1}{p^{2i-1}}+\frac{1}{p^{m-1}(p-1)} & \displaystyle -m\cdot\ord_p(\beta) + \sum_{i=1}^{m/2}\frac{1}{p^{2i-1}}+ \frac{1}{p^{m-1}(p-1)}\\
            \end{pmatrix}
        \end{align*}
    \end{corollary}
    \begin{proof}
        By \cite[Proposition~2.5]{BL-mathZ}, $\det(M_{\log})$ and $ \ell_0/X$ differ by a unit in  $\Lambda(\Gamma^\ac)$. It follows that $\frac{\ell_0}{X} M_{\log}^{-1}S$ is the adjugate matrix of $\det(S)S^{-1}M_{\log} = (\alpha-\beta)S^{-1}M_{\log}$, up to a unit. By Lemmas \ref{lem:log-matrix} and  \ref{lem:period-invertible}, we see that, up to a unit, the matrix of $\tilde{\Omega}^\epsilon_{V,1}$ given in Corollary~\ref{cor:explicit-matrices} specialized at $\Theta_m$, is of the form
        \begin{align*}
            \begin{pmatrix}
                1 & -a' \\ 0 & \epsilon_m
            \end{pmatrix}
            \begin{pmatrix}
                s_2/\alpha^m & s_2/\beta^m\\
                -s_1/\alpha^m & -s_1/\beta^m
            \end{pmatrix}.
        \end{align*}
        It follows from our assumption that $v\ge\frac{1}{p+1}$, for an even integer $m$ that is sufficiently large, 
        \begin{align*}
            \ord_p(s_2) = v+ \sum_{1\le i<m/2}\frac{1}{p^{2i}} = v + \frac{1-1/p^{m}}{p^2-1} > \frac{p-1/p^{m-1}}{p^2-1} = \ord_p(s_1).
        \end{align*}
        Since $a'$ is a $p$-adic unit, we have $\ord_p(s_2+a's_1) = \ord_p(s_1)$. Hence, the result follows.
    \end{proof}
\subsection*{Data availability statement}
Data sharing not applicable to this article as no datasets were generated or analysed during the current study.

\subsection*{Conflict of interest statement}
The authors have no conflicts of interest to declare that are relevant to the content of this article.

\bibliographystyle{amsalpha}
\bibliography{references}
\end{document}